\documentclass[11pt]{amsart} 
\usepackage{amsmath}
\usepackage{amssymb}
\usepackage[usenames]{color}
\newtheorem{theorem}{Theorem}[section]
\newtheorem{lemma}[theorem]{Lemma}
\newtheorem{proposition}[theorem]{Proposition}

\theoremstyle{definition}
\newtheorem{definition}[theorem]{Definition}

\theoremstyle{remark}

\newcommand{\ra}{\rightarrow}

\newcommand{\ul} {\underline}
\newcommand{\bl} {\begin{lemma}}
\newcommand{\el} {\end{lemma}}
\newcommand{\bt} {\begin{theorem}}
\newcommand{\et} {\end{theorem}} 
\newcommand{\Ckk} {\mathcal C _{k+1}}
\newcommand{\Tkk} {\mathcal T _{k+1}}
\newcommand{\mc}{\mathcal}
\newcommand{\limn}{\lim_{n \ra \infty}}
\newcommand{\limk}{\lim_{k \ra \infty}}

\newcommand{\Tk}{\mathcal T _k}
\newcommand{\Ck}{\mathcal C _k}
\newcommand  {\IX}{\widehat X}
\newcommand {\IR}{\mathbb R}
\newcommand {\IN}{\mathbb N}

\newcommand{\htop} {h_{top}}
\newcommand{\bp}{\begin{proof}}
\newcommand{\ep}{\end{proof}}
\newcommand {\be}{\begin{equation}}
\newcommand  {\ee} {\end{equation}}
\newcommand  {\beq} {\begin{eqnarray*}}
\newcommand  {\eeq} {\end{eqnarray*}}

\newcommand  {\bd} {\begin{definition}}
\newcommand  {\ed} {\end{definition}}
\newcommand  {\xr} {X_\rho}
 
\begin{document}

\title[The Irregular Set has Full Topological Pressure]{The Irregular Set for Maps with the Specification Property has Full Topological Pressure}
\author{Dan Thompson, University of Warwick}
\begin{abstract}
Let $(X,d)$ be a compact metric space, $f:X \mapsto X$ be a continuous map with the specification property, and $\varphi: X \mapsto \IR$ a continuous function. We consider the set of points for which the Birkhoff average of $\varphi$ does not exist (which we call the irregular set for $\varphi$) and show that this set is either empty or carries full topological pressure (in the sense of Pesin and Pitskel). We formulate various equivalent natural conditions on $\varphi$ that completely describe when the latter situation holds and give examples of interesting systems to which our results apply but were not previously known. As an application, we show that for a suspension flow over a continuous map with specification, the irregular set carries full topological entropy.
\end{abstract}
\maketitle
For a compact metric space $(X, d)$, a continuous map $f:X \mapsto X$ and a continuous function $\varphi: X \mapsto \IR$, we define the irregular set for $\varphi$ to be
\be \label{ir}
\widehat X_\varphi := \left \{ x \in X : \lim_{n \ra \infty} \frac{1}{n} \sum_{i = 0}^{n-1} \varphi (f^i (x)) \mbox{ does not exist } \right \}.
\ee
The irregular set arises naturally in the context of multifractal analysis, where one decomposes a space X into the disjoint union
\[
X = \bigcup_{\alpha \in \IR } X_{\varphi, \alpha} \cup \widehat X_\varphi,
\]
where $X_{\varphi, \alpha}$ is the set of points for which the Birkhoff average of $\varphi$ is equal to $\alpha$. We begin a program to understand the topological pressure of the multifractal decomposition 
by focusing on the irregular set $\widehat X_{\varphi}$ (we will consider the topological pressure of the sets $X_{\varphi, \alpha}$ in future work). The motivation for proving multifractal analysis results where pressure is the dimension characteristic is twofold. Firstly, topological pressure is a non-trivial and natural generalisation of topological entropy, which is the standard dynamical dimension characteristic. 
Secondly, understanding the topological pressure of the multifractal decomposition allows us to prove results about the topological entropy of systems related to the original system, for example, suspension flows (see \S5).

Our main result (theorem \ref{theorem1}) is that when $f$ has the specification property, $\widehat X_{\varphi}$ carries full topological pressure or is the empty set. We give conditions on $\varphi$ which completely describe which of the two cases hold. 



The class of maps satisfying the specification property includes the time-$1$ map of the geodesic flow of compact connected negative curvature manifolds and certain quasi-hyperbolic toral automorphisms as well as any system which can be modelled by a topologically mixing shift of finite type (see \S 4 for details). 

The first to notice the phenomenon of the irregular set carrying full entropy were Pesin and Pitskel \cite{PP2} in the case of the Bernoulli shift on 2 symbols.
Barreira and Schmeling \cite{BS} studied the irregular set for a variety of uniformly hyperbolic systems using symbolic dynamics. They showed  that, for example, the irregular set of a generic H\"older continuous function on a conformal repeller has full entropy (and Hausdorff dimension). Our results apply to a more general class of systems and we consider irregular sets for continuous functions which are not H\"older. 

Takens and Verbitskiy have obtained multifractal analysis results for the class of maps with specification, using topological entropy as the dimension characteristic  \cite{TV}, \cite{TV2}. However, they do not consider the irregular set. 
Ercai, Kupper and Lin \cite{EKL} proved that the irregular set is either empty or carries full entropy for maps with the specification property. 
Our results were derived independently and include the result of \cite{EKL} as a special case. Our methods are largely inspired by those of Takens and Verbitskiy \cite{TV}. To the best of the author's knowledge, our result is the first about the pressure of the irregular set.

We apply our main result to show that the irregular set for a suspension flow over a map with specification has full topological entropy. By considering the `$u$-dimension' of the irregular set in the base, Barreira and Saussol \cite{BS3} proved analogous results which apply when the suspension is over a shift of finite type. They assume H\"older continuity of $\varphi$ and the roof function, whereas we require only continuity.

We expect that an analogue of our main theorem \ref{theorem1} holds for flows with the specification property, and that our current method of proof can be adapted to this setting (although we do not pursue this here). Such an approach would not cover every suspension flow to which our current results apply. In particular, a special flow (i.e. a suspension flow with constant roof function) over a map with specification never has the specification property itself, but is in the class of flows treated in \S5.
 
In \S 1, we take care of our preliminaries. In \S 2, we state our main results and key ideas of the proof. In \S 3, we prove our main theorem. In \S4, we describe examples of maps to which our results can be applied. In \S 5, we apply our main result to suspension flows.

\section{Preliminaries}
We give the definitions and fix the notation necessary to give a precise statement of our results, including topological entropy for non-compact sets and the specification property. Let $(X,d)$ be a compact metric space and $f:X \mapsto X$ a continuous map. Let $C(X)$ denote the space of continuous functions from $X$ to $\IR$, and $\varphi, \psi \in C(X)$. Let $S_n \varphi (x) := \sum_{i = 0}^{n-1} \varphi (f^i (x))$ and for $c >0$, let $\mbox{Var}(\varphi, c) := \sup \{ |\varphi(x) - \varphi(y)| : d(x,y)< c\}$. Let $\mathcal{M}_{f} (X)$ denote the space of $f$-invariant probability measures and $\mathcal{M}^e_{f} (X)$ denote those which are ergodic. If $X^\prime \subseteq X$ is an $f$-invariant subset, let $\mathcal{M}_{f} (X^\prime)$ denote the subset of $\mathcal{M}_{f} (X)$ for which the measures $\mu$ satisfy $\mu(X^\prime) =1$. We define, for later use, the empirical measures
\[
\delta_{x, n} = \frac{1}{n} \sum_{k=0}^{n-1} \delta_{f^k (x)},
\]
where $\delta_x$ is the Dirac measure at $x$.
\begin{definition}
Let $\epsilon > 0$.  Given $n > 0$ and a point $x \in X$, define the open $(n, \epsilon)$-ball at $x$ by
\[
\mathit{B}_{n}(x, \epsilon) = \{ y \in X : d( f^i(x), f^i(y)) < \epsilon \mbox{ for all } i = 0, \ldots , n-1\}.
\]
\end{definition}
Alternatively, let us define a new metric
\[
d_n (x, y) = \max \{ d( f^i(x), f^i(y)) : i = 0, 1, \ldots, n-1 \}.
\]
It is clear that $\mathit{B}_{n}(x, \epsilon)$ is the open ball of radius $\epsilon$ around $x$ in the $d_n$ metric, and that if $n \leq m$ we have $d_n (x, y) \leq d_m (x, y)$ and $\mathit{B}_{m}(x, \epsilon) \subseteq \mathit{B}_{n}(x, \epsilon)$. 
\begin{definition}
Let $Z \subset X$, $n \in \IN$ and $\epsilon > 0$. We say a set $\mc S \subset Z$ is an $(n, \epsilon)$ spanning set for $Z$ if for every $z \in Z$, there exists $x \in \mc S$ with $d_n (x, z) \leq \epsilon$. Let $N(Z, n, \epsilon)$ denote the smallest cardinality of a $(n, \epsilon)$ spanning set for $Z$. We say a set $\mc R \subset Z$ is an $(n, \epsilon)$ separated set for $Z$ if for every $x, y \in \mc R$, $d_n (x, y) > \epsilon$. Let $S(Z, n, \epsilon)$ denote the largest cardinality of a $(n, \epsilon)$ separated set for $Z$.
\end{definition}
See \cite{Wa} for the basic properties of spanning sets and seperated sets. 

\subsection{Definition of the topological pressure} \label{entropy}
Let $Z \subset X$ be an arbitrary Borel set, not necessarily compact or invariant. We use the definition of topological pressure as a characteristic of dimension type, due to Pesin and Pitskel. 
We consider finite and countable collections of the form $\Gamma = \{ B_{n_i}(x_i, \epsilon) \}_i$. 
For $\alpha \in \IR$, we define the following quantities:
\[
Q(Z,\alpha, \Gamma, \psi) = \sum_{B_{n_i}(x_i, \epsilon) \in \Gamma} \exp \left(-\alpha n_i +\sup_{x\in B_{n_i}(x_i, \epsilon)} \sum_{k=0}^{n_i -1} \psi(f^{k}(x)) \right),
\]
\[
M(Z, \alpha, \epsilon, N, \psi) = \inf_{\Gamma} Q(Z,\alpha, \Gamma, \psi),
\]
where the infimum is taken over all finite or countable collections of the form $\Gamma = \{ B_{n_i}(x_i, \epsilon) \}_i$ with $x_i \in X$ such that $\Gamma$ covers Z and $n_i \geq N$ for all $i = 1, 2, \ldots$. Define
\[
m(Z, \alpha, \epsilon, \psi) = \lim_{N \rightarrow \infty} M(Z, \alpha,\epsilon, N, \psi).
\]
The existence of the limit is guaranteed since the function $M(Z, \alpha,\epsilon, N)$ does not decrease with N. By standard techniques, we can show the existence of
\begin{displaymath}
P_Z (\psi, \epsilon) := \inf \{ \alpha : m(Z, \alpha, \epsilon, \psi) = 0\} = \sup \{ \alpha :m(Z, \alpha, \epsilon, \psi) = \infty \}.
\end{displaymath}
\begin{definition}
The topological pressure of $\psi$ on $Z$ is given by
\[
P_Z (\psi) = \lim_{\epsilon \ra 0} P_Z (\psi, \epsilon).
\]
\end{definition}
See \cite{Pe} for verification of well-definedness of the quantities $P_Z (\psi, \epsilon)$ and $P_Z (\psi)$.
\begin{theorem}
Topological pressure satisfies:

(1) $P_{Z_1} (\psi) \leq P_{Z_2} (\psi)$ if $Z_1 \subseteq Z_2 \subseteq X$;

(2) $P_Z (\psi) = \sup_{i \geq 1} P_{Z_i} (\psi)$ where $Z = \bigcup_{i \geq 1} Z_i$ for $i = 1, 2, \ldots$.
\end{theorem}
If $Z$ is compact and invariant, our definition agrees with the usual topological pressure as defined in \cite{Wa}. We denote the topological pressure of the whole space by $P_X^{classic} (\psi)$, to emphasise that we are dealing with the familiar compact, invariant definition.  

\subsection{The specification property}
We are interested in transformations $f$ of the following type: 
\begin{definition} \label{3a}
A continuous map $f: X \mapsto X$ satisfies the specification property if for all $\epsilon > 0$, there exists an integer $m = m(\epsilon )$ such that for any collection $\left \{ I_j = [a_j, b_j ] \subset \IN : j = 1, \ldots, k \right \}$ of finite intervals with $a_{j+1} - b_j \geq m(\epsilon ) \mbox{ for } j = 1, \ldots, k-1 $ and any $x_1, \ldots, x_k$ in $X$, there exists a point $x \in X$ such that
\begin{equation} \label{3a.02}
d(f^{p + a_j}x, f^p x_j) < \epsilon \mbox{ for all } p = 0, \ldots, b_j - a_j \mbox{ and every } j = 1, \ldots, k.
\end{equation}
\end{definition}
The original definition of specification, due to Bowen, was stronger.
\begin{definition} \label{3a.2}
We say $f: X \mapsto X$ satisfies Bowen specification if under the assumptions of definition \ref{3a} and for every $p \geq b_k - a_1 + m(\epsilon)$, there exists a periodic point $x \in X$ of least period $p$ satisfying (\ref{3a.02}).
\end{definition}
One can describe a map $f$ with specification intuititively as follows. For any set of points $x_1, \ldots, x_k$ in $X$, there is an $x \in X$ whose orbit follows the orbits of all the points $x_1, \ldots, x_k$. In this way, one can connect together arbitrary pieces of orbit. If $f$ has Bowen specification, $x$ can be chosen to be a periodic point of any sufficiently large period. 

One can verify that a map with the specification property is topologically mixing. The following converse result holds \cite{Bl}, a recent proof of which is available in \cite{Bu}.
\begin{theorem} [Blokh Theorem]
A topologically mixing map of the interval has Bowen specification.
\end{theorem}
A factor of a system with specification has specification. 
We give a survey of many interesting examples of maps with the specification property in \S \ref{examp}. 

We will actually study a weakening of the definition of specification as follows. Let $X^\prime \subseteq X$ be $f$-invariant (but not necessarily compact).
\begin{definition} \label{3a.3}
A continuous map $f: X \mapsto X$ satisfies the specification property on $X^\prime$ if for all $\epsilon > 0$, there exists an integer $m = m(\epsilon )$ such that for any collection $\left \{ I_j = [a_j, b_j ] \subset \IN : j = 1, \ldots, k \right \}$ of finite intervals with $a_{j+1} - b_j \geq m(\epsilon ) \mbox{ for } j = 1, \ldots, k-1 $ and any $x_1, \ldots, x_k$ in $X^\prime$, there exists a point $x \in X$ such that
\[
d(f^{p + a_j}x, f^p x_j) < \epsilon \mbox{ for all } p = 0, \ldots, b_j - a_j \mbox{ and every } j = 1, \ldots, k.
\]
\end{definition}
Our results generalise to this setting naturally with little extra difficulty in the proofs. Although we do not offer an application of this extra generality, we think that there may be examples of non-uniformly hyperbolic systems where definition \ref{3a.3} holds on an interesting subset but where definition \ref{3a} is not verifiable.


\subsection{Cohomology and the irregular set}
Let $\phi_1, \phi_2 \in C(X)$. We say $\phi_1$ is cohomologous to $\phi_2$
if they differ by a coboundary, i.e. there exists $h \in C(X)$ such that
\[
\phi_1 = \phi_2 + h - h \circ f.
\]
For a constant $c$, let $Cob (X, f, c)$ denote the space of functions cohomologous to $c$ and $\overline {Cob (X, f, c)}$ be the closure of $Cob (X, f, c)$ in the sup norm.

We recall that $\widehat X_\varphi$ is the irregular set for $\varphi$, defined at (\ref{ir}). By Birkhoff's ergodic theorem, $\mu(\widehat X_\varphi) = 0$ for all $\mu \in \mc M_f (X)$. The following lemma describes conditions equivalent to $\widehat X_\varphi$  being non-empty.
\bl \label{equiv}
When $f$ has specification, the following are equivalent:

(a) $\widehat X_\varphi$ is non-empty;

(b) $\frac{1}{n}S_n \varphi$ does not converge pointwise to a constant;

(c) $\inf_{\mu \in \mathcal{M}_{f} (X)} \int \varphi d \mu < \sup_{\mu \in \mathcal{M}_{f} (X)} \int \varphi d \mu$;

(d) $\inf_{\mu \in \mathcal{M}^e_{f} (X)} \int \varphi d \mu < \sup_{\mu \in \mathcal{M}^e_{f} (X)} \int \varphi d \mu$;

(e) $\varphi \notin \bigcup_{c \in \IR} \overline {Cob (X, f, c)}$;

(f) $\frac{1}{n}S_n \varphi$ does not converge uniformly to a constant.
\el
The argument for (c) $\iff$ (e) $\iff$ (f) was given to the author by Peter Walters and is sketched here. In fact, no assumption on $f$ other than continuity is required except to prove that (a) is implied by the other properties.


\begin{proof} [Proof of lemma \ref{equiv}]
We show the contrapositive of (e) $\Rightarrow$ (f). Suppose $\frac{1}{n}S_n \varphi$ converges uniformly to $c$. Define for $n \in \IN$
\[
h_n (x) = \frac{1}{n} \sum_{i=1}^{n-1} (n-i) \varphi (f^{i-1} x). 
\]
We can verify that $\varphi - \frac{1}{n} S_n \varphi = h_n - h_n \circ f$ and it follows that $\varphi \in \overline {Cob (X, f, c)}$. 
The contrapositive of (c) $\Rightarrow$ (e) is straight forward.
Now we prove (f) $\Rightarrow$ (c). Let $\mu_1 \in \mc M_f (X)$ and let $c := \int \varphi d \mu_1$. From (f), there exists $\epsilon > 0$ and sequences $n_k \ra \infty$ and $x_k \in X$ such that
\[
| \frac{1}{n_k} S_{n_k} \varphi (x_k) - c | > \epsilon.
\] 
Let $\nu_k = \delta_{x_k, n_k}$ and let $\mu_2$ be a limit point of the sequence $\nu_k$. Then $\mu_2 \in \mc M_f (X)$ and $\int \varphi d \mu_2 \neq c$, so we are done.

The contrapositive of (a) $\Rightarrow$ (f) is clearly true and (b) $\Rightarrow$ (f) is trivial. We use an ergodic decomposition argument for (c) $\Rightarrow$ (d). For (d) $\Rightarrow$ (b), we take $\mu_1, \mu_2 \in \mc M^e_f (X)$ such that $\int \varphi d \mu_1 < \int \varphi d \mu_2$. We can find $x_i$ such that $\frac{1}{n} S_n \varphi (x_i) \ra \int \varphi d \mu_i$ for $i = 1,2$ and we are done.

Direct proof of (c) $\Rightarrow$ (a) using the specification property is possible, however it is a corollary of our main theorem so we omit the proof. 
\end{proof}

We mention briefly the complement of the irregular set.
For $\alpha \in \IR$, we define 
\[
X_{\varphi, \alpha} = \left \{ x \in X : \lim_{n \ra \infty} \frac{1}{n} S_n \varphi (x) = \alpha \right \}.
\]
We define the multifractal spectrum for $\varphi$ to be $\mathcal L_\varphi := \{ \alpha \in \IR : X_{\varphi, \alpha} \neq \emptyset \}.$
When $f$ has the specification property, $\mathcal L_\varphi$ is a non-empty bounded interval \cite{TV} and $\mathcal L_\varphi = \{ \int \varphi d \mu : \mu \in \mathcal{M}_{f} (X) \}$. We omit the proof, since we are not focusing our attention on $\mc L_\varphi$. 

We deduce that for maps with specification, the conditions of lemma \ref{equiv} are equivalent to the non-empty bounded interval of values taken by $\mc L_\varphi$ not being equal to a single point. 

\section{Results}
We state our results and introduce the key technical tools of the proof.
\begin{theorem} \label {theorem0.1}
Let $(X,d)$ be a compact metric space and $f:X \mapsto X$ be a continuous map with the specification property. Assume that $\varphi \in C(X)$ satisfies $\inf_{\mu \in \mathcal{M}_{f} (X)} \int \varphi d \mu < \sup_{\mu \in \mathcal{M}_{f} (X)} \int \varphi d \mu$. Let $\widehat X_\varphi$ be the irregular set for $\varphi$ defined as in (\ref{ir}), then $P_{\widehat X_\varphi} (\psi) = P_X^{classic} (\psi)$ for all $\psi \in C(X)$.
\end{theorem}
We remark that lemma \ref{equiv} provides us with other natural interpretations of the assumption $\inf_{\mu \in \mathcal{M}_{f} (X)} \int \varphi d \mu < \sup_{\mu \in \mathcal{M}_{f} (X)} \int \varphi d \mu$. We state the assumption in this way because it is natural for the method of proof. If our assumption fails, then $\widehat X_\varphi = \emptyset$.

In fact, we prove a slightly stronger version of the theorem.
\begin{theorem} \label {theorem1}
Let $(X,d)$ be a compact metric space, $f:X \mapsto X$ be a continuous map and $X^\prime \subseteq X$ be $f$-invariant. Assume $f$ satisfies the specification property on $X^\prime$. Assume that $\varphi \in C(X)$ satisfies $\inf_{\mu \in \mathcal{M}_{f} (X^\prime)} \int \varphi d \mu < \sup_{\mu \in \mathcal{M}_{f} (X^\prime)} \int \varphi d \mu$. Let $\widehat X_\varphi$ be the irregular set for $\varphi$ defined as in (\ref{ir}), then for all $\psi \in C(X)$,
\[
P_{\widehat X_\varphi} (\psi) \geq \sup \left \{ h_\mu + \int \psi d \mu : \mu \in \mathcal M_f (X^\prime) \right \}.
\]
If $\sup \left \{ h_\mu + \int \psi d \mu : \mu \in \mathcal M_f (X^\prime) \right \} = P_X^{classic} (\psi)$, then we have $P_{\widehat X_\varphi} (\psi) = P_X^{classic} (\psi)$.
\end{theorem}
If $\mathcal{M}_{f} (X^\prime)$ is dense in $\mathcal{M}_{f} (X)$, we need only assume $\inf_{\mu \in \mathcal{M}_{f} (X)} \int \varphi d \mu < \sup_{\mu \in \mathcal{M}_{f} (X)} \int \varphi d \mu$. We adapt an ingenious method of Takens and Verbitskiy, which can be found in \S5 of \cite{TV} and was in turn developed from a large deviations proof of Young \cite{Yo}. The key ingredients for the Takens and Verbitskiy proof are an application of the Entropy Distribution Principle \cite{TV} and Katok's formula for measure-theoretic entropy \cite{K}. We are required to generalise both. We offer two generalisations of the Entropy Distribution Principle. While the first offers a more straight forward generalisation, we will use the second as it offers us a short cut in the proof later on. We prove only the second, as the proof of the first is similar.
\begin{proposition}[Pressure distribution principle]
Let $f : X \mapsto X$ be a continuous transformation. Let $Z \subseteq X$ be an arbitrary Borel set. Suppose there exists a constant $s \geq 0$ such that for sufficiently small $\epsilon > 0$ one can find a Borel probability measure $\mu_\epsilon$ and a constant $K(\epsilon ) > 0$ satisfying $\mu_\epsilon (Z) > 0$ and $\mu_\epsilon (B_n (x, \epsilon)) \leq K(\epsilon ) \exp \{-ns + \sum_{i=0}^{n-1} \psi(f^i x)\}$ for sufficiently large $n$ and every ball $B_n (x, \epsilon)$ which has non-empty intersection with $Z$. Then $P_Z (\psi) \geq s$.
\end{proposition}
\begin{proposition}[Generalised pressure distribution principle] \label {theorem3}
Let $f : X \mapsto X$ be a continuous transformation. Let $Z \subseteq X$ be an arbitrary Borel set. Suppose there exists $\epsilon > 0$ and $s \geq 0$ such that one can find a sequence of Borel probability measures $\mu_k$, a constant $K> 0$, and a limit measure $\nu$ of the sequence $\mu_{k}$ satisfying $\nu(Z) > 0$ such that 
\[
\limsup_{k \ra \infty} \mu_{k} (B_n (x, \epsilon)) \leq K \exp \{-ns + \sum_{i=0}^{n-1} \psi(f^i x)\}
\] 
for sufficiently large $n$ and every ball $B_n (x, \epsilon)$ which has non-empty intersection with $Z$. Then $P_Z (\psi, \epsilon) \geq s$.
\end{proposition}
\begin{proof}
Choose $\epsilon > 0$ and measure $\nu$ satisfying the conditions of the theorem. Let $\Gamma = \{ B_{n_i} (x_i, \epsilon) \}_i$ cover $Z$ with all $n_i$ sufficiently large. We may assume that $B_{n_i} (x_i, \epsilon) \cap Z \neq \emptyset$ for every $i$. Then 
\begin{eqnarray*}
Q (Z, s, \Gamma, \psi) & = & \sum_i \exp \left \{-sn_i + \sup_{y \in B_{n_i} (x_i, \epsilon)} \sum_{k=0}^{n_i -1} \psi(f^{k}(y)) \right \}\\
& \geq & \sum_i \exp \left \{-sn_i + \sum_{k=0}^{n_i -1} \psi(f^{k}(x_i)) \right \}\\
&\geq & \ K^{-1} \sum_i  \limsup_{k \ra \infty} \mu_{k} (B_n (x_i, \epsilon)) \\
& \geq & \ K^{-1} \sum_i \nu (B_n (x_i, \epsilon)) \geq K^{-1} \nu (Z) > 0
\end{eqnarray*}
So $M (Z, s, \epsilon, \psi) > 0$ and thus $P_Z (\psi, \epsilon) \geq s$. 
\end{proof}
The following result generalises Katok's formula for measure-theoretic entropy. In \cite{Me}, Mendoza gave a proof based on ideas from the Misiurewicz proof of the variational principle. Although he states the result under the assumption that $f$ is a homeomorphism, his proof works for $f$ continuous.
\begin{proposition} \label {theorem4}
Let $(X,d)$ be a compact metric space, $f:X \mapsto X$ be a continuous map and $\mu$ be an ergodic invariant measure. For $\epsilon > 0$, $\gamma \in (0, 1)$ and $\varphi \in C(X)$, define 
\[
N^\mu (\psi, \gamma, \epsilon, n) = \inf \left \{ \sum_{x \in S} \exp \left \{ \sum_{i=0}^{n -1} \psi(f^{i}x)\right \} \right \}
\] 
where the infimum is taken over all sets $S$ which $(n, \epsilon)$ span some set $Z$ with $\mu(Z) \geq 1 - \gamma$. We have
\[
h_\mu + \int \psi d \mu  = \lim_{\epsilon \ra 0} \liminf_{n \ra \infty} \frac{1}{n} \log N^\mu (\psi, \gamma, \epsilon, n).
\]
The formula remains true if we replace the $\liminf$ by $\limsup$. 
\end{proposition}

We now begin the proof of theorem \ref{theorem1}. For the sake of clarity, it will be convenient to give the proof under a certain additional hypothesis, which we will later explain how to remove. 
\begin{theorem} \label {theorem5}
Let us assume the hypotheses of theorem \ref{theorem1} and fix $\psi \in C(X)$. Let
\[C := \sup \left \{ h_\mu + \int \psi d \mu : \mu \in \mathcal M_f (X^\prime) \right \}.\] 
Let us assume further that $P^{classic}_X (\psi)$ is finite and for all $\gamma > 0$, there exist ergodic measures $\mu_1, \mu_2 \in \mathcal M_f (X^\prime)$ which satisfy

(1) $h_{\mu_i} + \int \psi d \mu_i > C - \gamma$ for $i = 1, 2$,

(2) $\int \varphi d \mu_1 \neq \int \varphi d\mu_2.$\\
Then $P_{\widehat X_\varphi} (\psi) \geq C$. If $C = P^{classic}_X (\psi)$, for example when $X^\prime = X$, then $P_{\widehat X_\varphi} (\psi) = P^{classic}_X (\psi)$.
\end{theorem}


The assumption that $P^{classic}_X (\psi)$ is finite is trivial to remove and is included only for notational convenience. Given a result from \cite{EKW}, we give a short proof that the hypotheses of theorem \ref{theorem0.1} imply those of theorem \ref{theorem5} when the map $\mu \ra h_\mu$ is upper semi-continuous. We explain how to modify the proof of theorem \ref{theorem5} to obtain a self contained proof of theorem \ref{theorem1} in \S \ref{modif}.
\begin{proof} [Proof of theorem \ref {theorem0.1}.]
Let $\mu_1$ be ergodic and satisfy $h_{\mu_1} + \int \psi d \mu_1 > C - \gamma /3$, Let $\nu \in \mc M_f (X)$ satisfy $\int \varphi d {\mu_1} \neq \int \varphi d \nu$. Let $\nu^\prime = t \mu_1 + (1-t) \nu$ where $t \in (0,1)$ is chosen sufficiently close to $1$ so that $h_{\nu^\prime} + \int \psi d \nu^\prime > C - 2 \gamma /3$. By theorem B of \cite{EKW}, when $f$ has the specification property and the map $\mu \ra h_\mu$ is upper semi-continuous, we can find a sequence of ergodic measures $\nu_n \in \mc M_f (X)$ such that $h_{\nu_n} \ra h_{\nu^\prime}$ and $\nu_n \ra \nu^\prime$ in the weak-$\ast$ topology. Therefore, we can choose a measure belonging to this sequence which we call $\mu_2$ which satisfies $h_{\mu_2} + \int \psi d \mu_2 > C - \gamma$ and $\int \varphi d \mu_1 \neq \int \varphi d\mu_2.$
\end{proof}

\section{Proof of the main theorem \ref {theorem5}}
Let us fix a small $\gamma > 0$, and take the measures $\mu_1$ and $\mu_2$ provided by our hypothesis. Choose $\delta > 0$ sufficiently small so
\[
\left | \int \varphi d \mu_1 - \int \varphi d\mu_2 \right | > 4 \delta.
\]
Let $\rho : \IN \mapsto \{1, 2\}$ be given by $\rho(k) = 1 + k$ $ (mod 1)$. Choose a strictly decreasing sequence $\delta_k \ra 0$ with $\delta_1 < \delta$ and a strictly increasing sequence $l_k \ra \infty$ so the set
\be \label{5}
Y_k := \left \{ x \in X^\prime : \left | \frac{1}{n} S_n \varphi (x) - \int \varphi d\mu_{\rho(k)} \right | < \delta_k \mbox{ for all } n \geq l_k\right \}
\ee
satisfies $\mu_{\rho(k)} (Y_k) > 1 - \gamma$ for every $k$. This is possible by Birkhoff's ergodic theorem. 

The following lemma follows readily from proposition \ref{theorem4}. 
\begin{lemma} \label {lemma1}
For any sufficiently small $\epsilon > 0$, we can find a sequence $n_k \ra \infty$ and a countable collection of finite sets $\mc S_k$ so that each $\mc S_k$ is an $(n_k, 4 \epsilon)$ separated set for $Y_k$ and $M_k := \sum_{x \in \mc S_k} \exp \left \{ \sum_{i=0}^{n_k -1} \psi(f^{i}x)\right \}$ satisfies
\[
M_k \geq \exp( n_k (C - 4 \gamma)).
\]
Furthermore, the sequence $n_k$ can be chosen so that $n_k \geq l_k$ and $n_k \geq 2^{m_{k}}$, where $m_k = m(\epsilon/2^k)$ is as in definition \ref{3a.3} of the specification property.
\end{lemma}
\begin{proof}
By proposition \ref {theorem4}, let us choose $\epsilon$ sufficiently small so
\[
\liminf_{n \ra \infty} \frac{1}{n} \log N^{\mu_i} (\psi, \gamma, 4 \epsilon, n) \geq  h_{\mu_i} + \int \psi d \mu_i - \gamma \geq C -2 \gamma \mbox{ for } i = 1, 2.
\]
For $A \subset X$, let
\[
Q_n (A, \psi, \epsilon) = \inf \left \{ \sum_{x \in S} \exp \left \{ \sum_{k=0}^{n -1} \psi(f^{k}x) \right \}: S \mbox{ is (n, $\epsilon$) spanning set for } A \right \}, 
\]
\[
P_n (A, \psi, \epsilon) = \sup \left \{ \sum_{x \in S} \exp \left \{ \sum_{k=0}^{n -1} \psi(f^{k}x) \right \}: S \mbox{ is (n, $\epsilon$) separated set for } A \right \}. 
\]
We have $Q_n (A, \psi, \epsilon) \leq P_n (A, \psi, \epsilon)$ and since $\mu_{\rho(k)} (Y_k) > 1 - \gamma$ for every $k$, it is immediate that
\[
Q_n(Y_k, \psi, 4 \epsilon) \geq N^{\mu_{\rho(k)}} (\psi, \gamma, \epsilon, n).
\]
Let $M(k,n) = P_n(Y_k, \psi, 4 \epsilon)$. For each $k$, we obtain
\[
\liminf_{n \ra \infty} \frac{1}{n} \log M (k, n) \geq \liminf_{n \ra \infty} \frac{1}{n} \log N^{\mu_{\rho(k)}} (\psi, \gamma, 4 \epsilon, n) \geq C -2 \gamma.
\]
We may now choose a sequence $n_k \ra \infty$ satisfying the hypotheses of the lemma so 
\[
\frac{1}{n_k} \log M(k, n_k) \geq C - 3 \gamma.
\]
Now for eack $k$, let $\mc S_k$ be a choice of $(n_k, 4 \epsilon)$ separated set for $Y_k$ which satisfies 
\[
\frac{1}{n_k} \log \left \{\sum_{x \in \mc S_k} \exp \left \{ \sum_{i=0}^{n -1} \psi(f^{i}x) \right \} \right \} \geq \frac{1}{n_k} \log M(k, n_k) - \gamma.
\]
Let $M_k := \sum_{x \in \mc S_k} \exp \left \{ \sum_{i=0}^{n -1} \psi(f^{i}x) \right \}$, then
\[
\frac{1}{n_k} \log M_k \geq \frac{1}{n_k} \log M(k, n_k) - \gamma \geq C - 4 \gamma.
\]
We rearrange to obtain the desired result.
\end{proof}
We choose $\epsilon$ sufficiently small so that $Var(\psi, 2 \epsilon) < \gamma$ and $Var(\varphi, 2 \epsilon) < \delta$, and fix all the ingredients provided by lemma \ref{lemma1}.

Our strategy is to construct a certain fractal $F \subset \IX_\varphi$, on which we can define a sequence of measures suitable for an application of the generalised pressure distribution principle. 

\subsection{Construction of the fractal F}
We begin by constructing two intermediate families of finite sets. The first such family we denote by $\{ \mathcal C_k \}_{k \in \IN}$ and consists of points which shadow a very large number $N_k$ of points from $\mc S_k$. The second family we denote by $\{ \Tk\}_{k \in \IN}$  and  consist of points which shadow points (taken in order) from ${\mathcal C _1}, {\mathcal C _2}, \ldots, {\mathcal C _k}$. We choose $N_k$ to grow to infinity very quickly, so the ergodic average of a point in $\Tk$ is close to the corresponding point in $\Ck$.

\subsubsection{Construction of the intermediate sets $\{ \Ck \}_{k \in \IN}$} \label{f}
Let us choose a sequence $N_k$ which increases to $\infty$ sufficiently quickly so that
\begin{equation} \label{f.1}
\limk \frac{n_{k+1} + m_{k+1}}{N_k} = 0,  \limk \frac{N_1 (n_1 + m_1) + \ldots + N_k (n_k + m_k)}{N_{k+1}} = 0.
\end{equation}
We enumerate the points in the sets $\mc S_k$ provided by lemma \ref{lemma1} and write them as follows
\[
\mathcal S_k = \{ x_i^k : i = 1, 2, \ldots, \# \mc S_k \}.
\]
Let us make a choice of $k$ and consider the set of words of length $N_k$ with entries in $\{ 1, 2, \ldots, \# \mc S_k\}$. Each such word $ \ul i = (i_1, \ldots, i_{N_k} )$ represents a point in $\mathcal S_k ^{N_k}$. Using the specification property, 
 we can choose a point $y:= y(i_1, \ldots, i_{N_k} )$ which satisfies
\[
d_{n_k}(x_{i_j}^k, f^{a_j} y) < \frac{\epsilon}{2^k}
\]
for all $j \in \{1, \ldots, N_k \}$ where $a_j = (j-1)(n_k +m_k)$. (i.e. $y$ shadows each of the points $x_{i_j}^k$ in order for length $n_k$ and gap $m_k$.) We define
\[
\Ck = \left \{ y(i_1, \ldots, i_{N_k}) \in X : (i_1, \ldots, i_{N_k}) \in \{1, \ldots, \# \mc S_k \}^{N_k} \right \}.
\]
Let $c_k = N_k n_k + (N_k - 1) m_k$. Then $c_k$ is the amount of time for which the orbit of points in $\Ck$ has been prescribed. It is a corollary of the following lemma that distinct sequences $(i_1, \ldots, i_{N_k} )$ give rise to distinct points in $\Ck$. Thus the cardinality of $\Ck$, which we shall denote by $\# \Ck$, is $\#S_k^{N_k}$.
\bl \label{lemma2}
Let $\ul i$ and $\ul j$ be distinct words in $\{ 1, 2, \ldots M_k \}^{N_k}$. Then $y_1 := y ( \ul i )$ and $y_2 := y ( \ul j )$ are $(c_k, 3 \epsilon)$ separated points (ie. $d_{c_k} ( y_1, y_2 ) > 3 \epsilon$).
\el
\begin{proof} 
Since $\ul i \neq \ul j$, there exists $l$ so $i_l \neq j_l$. We have
\[
d_{n_k} (x_{i_l}^k, f^{a_l} y_1 ) < \frac{\epsilon}{2^k}, d_{n_k} (x_{j_l}^k, f^{a_l} y_2 ) < \frac{\epsilon}{2^k} \mbox{ and } d_{n_k} (x_{i_l}^k, x_{j_l}^k ) > 4 \epsilon.
\] 
Combining these inequalities, we have
\begin{eqnarray*}
d_{c_k} (y_1, y_2) & \geq & d_{n_k} (f^{a_l} y_1, f^{a_l} y_2 ) \\
	      & \geq &  d_{n_k} (x_{i_l}^k, x_{j_l}^k ) - d_{n_k} (x_{i_l}^k, f^{a_l} y_1 ) - d_{n_k} (x_{j_l}^k, f^{a_l} y_2 ) \\ & > & 4 \epsilon - \epsilon /2 - \epsilon /2 = 3 \epsilon.
\end{eqnarray*}
\end{proof}
\subsubsection{Construction of the intermediate sets $\{ \Tk \}_{k \in \IN}$} \label{g}
We use the specification property to construct points whose orbits shadow points (taken in order) from ${\mathcal C _1}, {\mathcal C _2}, \ldots, {\mathcal C _k}$. Formally, we define $\Tk$ inductively. Let $\mathcal T _1 = \mathcal C _1$. We construct $\mc T_{k+1}$ from $\mc T_{k}$ as follows. Let $x \in \Tk$ and $y \in \mc C_{k+1}$. Let $t_1 = c_1$ and $t_{k+1} = t_k + m_{k+1} + c_{k+1}$. Using specification, we can find a point $z := z(x, y)$ which satisfies
\[
d_{t_k}(x, z) < \frac{\epsilon}{2^{k+1}} \mbox{ and } d_{c_{k+1}}(y, f^{t_k +m_{k+1}}z) <\frac{\epsilon}{2^{k+1}} .
\]
Define $\mc T_{k+1} = \{ z(x, y) : x \in \Tk, y \in \mc C_{k+1} \}$.
Note that $t_k$ is the amount of time for which the orbit of points in $\Tk$ has been prescribed. Once again, points constructed in this way are distinct. So we have
\[
\# \Tk = \# \mc C_1 \ldots  \# \mc C_k= \#S_1^{N_1} \ldots \# S_k^{N_k}.
\]
This fact is a corollary of the following straight forward lemma:
\begin{lemma} \label{lemma3}
For every $x \in \Tk$ and distinct $y_1, y_2 \in \mc C_{k+1}$
\[
d_{t_k} (z(x,y_1), z(x, y_2)) < \frac{\epsilon}{2^k} \mbox{ and } d_{t_{k+1}} (z(x,y_1), z(x, y_2)) > 2 \epsilon.
\]
Thus $\Tk$ is a $(t_k, 2 \epsilon)$ separated set. In particular, if $z, z' \in \Tk$, then \[
\overline B_{t_{k}} (z, \frac{\epsilon}{2^k}) \cap \overline B_{t_{k}} (z', \frac{\epsilon}{2^k}) = \emptyset.\]
\end{lemma}
\begin{proof} 
Let $p := z(x, y_1)$ and $q := z(x, y_2)$. The first inequality is trivial since by construction, $d_{t_k}(x, z_i) < \epsilon /2^{k+1}$ for $ i = 1,2$. 

Using lemma \ref {lemma2}, we obtain the second inequality as follows:
\begin{eqnarray*}
d_{t_{k+1}} (p, q) & \geq & d_{c_{k+1}} (f^{t_k + m_{k+1}} p, f^{t_k + m_{k+1}} q ) \\ & \geq &  d_{c_{k+1}} (y_1, y_2 ) - d_{c_{k+1}} (y_1, f^{t_k + m_{k+1}} p ) - d_{c_{k+1}} (y_2, f^{t_k + m_{k+1}} q ) \\ & > & 3\epsilon - \epsilon /2 - \epsilon /2 = 2 \epsilon.
\end{eqnarray*}
The third statement is a straightforward consequence of the second.
\end{proof}
Following the terminology of Takens and Verbitskiy, we say $z \in \Tkk$ descends from $x \in \Tk$ if $z = z(x, y)$ for some $y \in \Ckk$.
\bl \label{lemma4}
If $z \in \Tkk$ descends from $x \in \Tk$ then
\[
\overline B_{t_{k+1}} (z, \frac{\epsilon}{2^k}) \subset \overline B_{t_{k}} (x, \frac{\epsilon}{2^{k-1}}).
\]
\el
\begin{proof}
Let $z' \in \overline B_{t_{k+1}} (z, \frac{\epsilon}{2^k})$. Then
\begin{eqnarray*}
d_{t_{k}} (z', x) & \leq & d_{t_{k+1}} (z', z) + d_{t_{k}} (z, x)\\
	      & \leq &  \epsilon /2^k + \epsilon /2^{k+1} \leq \epsilon /2^{k-1}.
\end{eqnarray*} 
\end{proof}
\subsubsection{Construction of the fractal $F$ and a special sequence of measures $\mu_k$} \label{ss}
Let $F_k = \bigcup_{x \in \Tk} \overline B_{t_k} (x, \frac{\epsilon}{2^{k-1}})$. By lemma \ref{lemma4}, $F_{k+1} \subset F_k$. Since we have a decreasing sequence of connected compact sets, the intersection $F = \bigcap_k F_k$ is non-empty. Further, every point $p \in F$ can be uniquely represented by a sequence $\underline p = (\ul p_1, \ul p_2, \ul p_3,. \ldots )$ where each $\ul p_i = (p_1^i, \ldots, p_{N_i}^i) \in \{ 1, 2, \ldots M_i \}^{N_i}$. Each point in $\Tk$ can be uniquely represented by a finite word $(\ul p_1, \ldots \ul p_k)$. We introduce some useful notation to help us see this. Let $y(\ul p_i) \in \mc C_i$ be defined as in \ref{f}. Let $z_1 (\ul p ) = y(\ul p_1)$ and proceeding inductively, let $z_{i+1} ( \ul p) = z (z_{i} ( \ul p), y(\ul p_{i+1})) \in \mc T _{i+1}$ be defined as in \ref{g}. We can also write $z_i (\ul p)$ as $z(\ul p_1, \ldots, \ul p_i)$.  Then define $p := \pi \ul p$ by
\[
p = \bigcap_{i\in \IN} \overline B_{t_{i}} (z_{i} ( \ul p), \frac{\epsilon}{2^{i-1}}).
\]
It is clear from our construction that we can uniquely represent every point in $F$ in this way.
\bl \label{dist}
Given $z = z(\ul p_1, \ldots, \ul p_k) \in \Tk$, we have for all $i \in \{1, \ldots, k\}$ and all $l \in \{1, \ldots, N_i\}$,
\[
d_{n_i}(x^i_{p^i_{l}}, f^{t_{i-1} + m_{i-1} + (l-1)(m_i + n_i)}z) < 2 \epsilon.
\]
\el
\bp
We fix $i \in \{1, \ldots, k\}$ and $l \in \{1, \ldots, N_i\}$. For $m \in \{1, \ldots, k-1\}$, let $z_m = z(\ul p_1, \ldots, \ul p_m) \in \mc T_m$. Let $a = t_{i-1} + m_{i-1}$ and $b = (l-1)(m_i + n_i)$. Then
\[
d_{n_i} (x^i_{p^i_l}, f^{a+b} z) < d_{n_i} (x^i_{p^i_l}, f^b y(\ul p_i)) + d_{n_i} (f^b y(\ul p_i), f^{a+b} z_i) + d_{n_i} (f^{a+b} z_i, f^{a+b} z).
\]
We have, by construction,
\[
d_{n_i} (x^i_{p^i_l}, f^{b} y(\ul p_i) ) < \frac{\epsilon}{2^i}.
\]
We have, by construction,
\[
d_{n_i} (f^b y(\ul p_i), f^{a+b} z_i) \leq d_{c_i} (y(\ul p_i), f^a z) < \frac{\epsilon}{2^{i+1}}.
\]
We have
\beq
d_{n_i} (f^{a+b} z_i, f^{a+b} z) < d_{t_i} (z_i, z) &<& d_{t_i} (z_i, z_{i+1}) + \ldots + d_{t_i} (z_{k-1}, z)\\
&<& \frac{\epsilon}{2^{i+1}}+ \frac{\epsilon}{2^{i+2}} +\ldots + \frac{\epsilon}{2^{k}}.
\eeq
Combining the inequalities, we obtain $d_{n_i} (f^{a+b} z, x^i_{p^i_l}) < \sum_{m=i}^{k} \frac{\epsilon}{2^{m}} + \frac{\epsilon}{2^{i+1}} <2 \epsilon$, as required.
\ep
We now define the measures on $F$ which yield the required estimates for the Pressure Distribution Principle. For each $z \in \Tk$, we associate a number $\mc L (z) \in (0, \infty)$. Using these mumbers as weights, we define, for each $k$, an atomic measure centred on $\Tk$. Precisely, if $z = z(\ul p_1, \ldots \ul p_k)$, we define
\[
\mc L(z) := \mc L(\ul p_1) \ldots \mc L(\ul p_k),
\]
where if $\ul p_i = (p^i_{1}, \ldots, p^i_{N_i})\in \{1, \ldots, \# \mc S_i \}^{N_i}$, then
\[
\mc L(\ul p_i) := \prod_{l=1}^{N_i} \exp S_{n_i} \psi (x^i_{p^i_l}). 
\]
We define
\[
\nu_k := \sum_{z \in \Tk}\delta_z \mc L(z). 
\]
We normalise $\nu_k$ to obtain a sequence of probability measures $\mu_k$. More precisely, we let $\mu_k := \frac{1}{\kappa_k} \nu_k$, where $\kappa_k$ is the normalising constant
\[
\kappa_k := \sum_{z \in \Tk} \mc L_k (z).
\]
\bl
$\kappa_k = M_1^{N_1} \ldots M_k^{N_k}$.
\el
\bp
We note that
\beq
\sum_{\ul p_i \in \{1, \ldots, \# \mc S_i \}^{N_i}} \mc L (\ul p_i) &=& \sum_{p^i_1 = 1}^{\# \mc S_i} \exp S_{n_i} \psi (x^i_{p^1_l}) \ldots \sum_{p^i_{N_i} = 1}^{\# \mc S_i} \exp S_{n_i} \psi (x^i_{p^i_{N_i}})\\ 
&=& M_i^{N_i}
\eeq
By the definition and since each $z \in \Tk$ corresponds uniquely to a sequence $(\ul p_1, \ldots, \ul p_k)$, we have
\[
\sum_{z \in \mc T_k} \mc L_k (z) = \sum_{\ul p_1 \in \{1, \ldots, \# \mc S_1 \}^{N_1}} \ldots \sum_{\ul p_k \in \{1, \ldots, \# \mc S_k \}^{N_k}} \mc L (\ul p_1) \ldots \mc L (\ul p_k).
\]
The result follows.
\ep
\bl \label{lemma5}
Suppose $\nu$ is a limit measure of the sequence of probability measures $\mu_k$. 
Then $\nu (F) = 1$.
\bp
Suppose $\nu$ is a limit measure of the sequence of probability measures $\mu_k$. Then $\nu = \lim_{k \ra \infty} \mu_{l_k}$ for some $l_k \ra \infty$.
For any fixed $l$ and all $p \geq 0$, $\mu_{l+p} (F_{l}) = 1$ since $\mu_{l+p} (F_{l+p}) = 1$ and $F_{l+p} \subseteq F_{l}$. Therefore, $\nu(F_l) \geq \limsup_{k \ra \infty} \mu_{l_k} (F_l) = 1$. It follows that $\nu (F) = \lim_{l \ra \infty} \nu (F_{l}) = 1$.
\ep
\el
In fact, the measures $\mu_k$ converge. However, by using the generalised pressure distribution principle, we do not need to use this fact and so we omit the proof (which goes like lemma 5.4 of \cite{TV}).

We verify that $F \subset \widehat X_\varphi$.
\bl \label{lemma6}
For any $p \in F$, the sequence $\frac{1}{t_k} \sum_{i = 0}^{t_k-1} \varphi (f^i (p))$ diverges. 
\el 
\begin{proof}
Let us choose a point $p \in F$. Using the notation of \ref{ss}, let $y_k := y( \ul p_k)$ and $z_k = z_k (\ul p)$. We first show that
\begin{equation} \label{i}
\left | \frac{1}{c_k} S_{c_k} \varphi (y_k) - \int \varphi d\mu_{\rho(k)} \right | \ra 0.
\end{equation}
We rely on the fact that $\mbox{Var}( \varphi, c) \ra 0$ as $c \ra 0$ and that
\be \label{j}
\limk \frac{n_k N_k}{c_k} = 1, \limk \frac{m_k (N_k - 1)}{c_k} = 0 \mbox{ and } \limk \delta_k = 0.
\ee
The first two limits follow from the assumption that $n_k \geq 2^{m_k}$. Let $a_j = (j-1)(n_k+m_k)$. We have
\beq
\left | S_{c_k} \varphi (y_k) - c_k \int \varphi d\mu_{\rho(k)} \right | 
 &\leq& \left | \sum_{j = 1}^{N_k} S_{n_k} \varphi (f^{a_j} y_k) - c_k \int \varphi d\mu_{\rho(k)} \right |\\ &+& m_k (N_k - 1) \| \varphi \|
\eeq
\beq
\leq \left | \sum_{j = 1}^{N_k} S_{n_k} \varphi (f^{a_j} y_k) - \sum_{j = 1}^{N_k} S_{n_k} \varphi (x_{i_j}^k) \right | &+& \left | \sum_{j = 1}^{N_k} S_{n_k} \varphi (x_{i_j}^k) - c_k \int \varphi d\mu_{\rho(k)} \right | \\&+& m_k (N_k - 1) \| \varphi \|
\eeq
\beq
\leq \sum_{j = 1}^{N_k} \left | S_{n_k} \varphi (f^{a_j} y_k) - S_{n_k} \varphi (x_{i_j}^k) \right | &+&  \sum_{j = 1}^{N_k} \left | S_{n_k} \varphi (x^k_{i_j}) - n_k \int \varphi d\mu_{\rho(k)} \right | \\&+& m_k (N_k - 1) \{ \| \varphi \| +\int \varphi d\mu_{\rho(k)} \}
\eeq
\beq \label{k}
\leq n_k N_k \{ \mbox{Var}( \varphi, \epsilon/2^k) +  \delta_k \} + m_k (N_k - 1) \{ \| \varphi \| + \int \varphi d\mu_{\rho(k)} \}.
\eeq
We have used the fact $d_{n_k}(x_{i_j}^k, f^{a_j} y_k) < \epsilon /2^k$ in the last line. The statement of (\ref{i}) follows from this and (\ref{j}). 

Let $p' = f^{t_k-c_k} p$ and $z'_k = f^{t_k-c_k} z_k$. Using $d_{t_k} (p, z_k) \leq \epsilon / 2^{k-1}$, we have 
\begin{eqnarray*}
d_{c_{k}} (p', y_k) & \leq & d_{c_k} (p', z_k') + d_{c_k} (z_k', y_k)\\
	      & \leq &  \epsilon /2^{k-1} + \epsilon /2^{k} \leq \epsilon /2^{k-2}.
\end{eqnarray*}
Using this and (\ref{i}), we obtain
\begin{equation} \label{m}
\left | \frac{1}{c_k} S_{c_k} \varphi (p') - \int \varphi d\mu_{\rho(k)} \right | \leq \mbox{Var}( \varphi, \epsilon/2^{k-2}).
\end{equation}
The final ingredient we require is to show that
\be \label{n}
\left | \frac{1}{t_k} S_{t_k} \varphi (p) - \frac{1}{c_k} S_{c_k} \varphi (p') \right | \ra 0.
\ee
From the assumptions of (\ref{f.1}), we can verify that $c_k/t_k \ra 1$. Thus for arbitrary $\gamma > 0$ and sufficiently large $k$, we have $| c_k/t_k - 1 | < \gamma$. We have
\beq
\left | \frac{1}{t_k} S_{t_k} \varphi (p) - \frac{1}{c_k} S_{c_k} \varphi (p') \right | & = &\left | \frac{1}{t_k} S_{t_k - c_k} \varphi (p) + \frac{1}{c_k} S_{c_k} \varphi (p')\left(\frac{c_k}{t_k} - 1\right) \right |\\
& \leq & \left | \frac{t_k - c_k}{t_k} \| \varphi \| + \gamma \frac{1}{c_k} S_{c_k} \varphi (p') \right |\\ 
& \leq & 2 \gamma \| \varphi \| 
\eeq
Since $\gamma$ was arbitrary, we have verified (\ref{n}). Using (\ref{m}) and (\ref{n}), it follows that $\left | \frac{1}{t_k} S_{t_k} \varphi (p) - \int \varphi d\mu_{\rho(k)} \right | \ra 0$.
\end{proof}
For an affirmative answer to theorem \ref{theorem5}, we give a sequence of lemmas which will allow us to apply the generalised pressure distribution principle.
Let $\mc B := B_n (q, \epsilon /2)$ be an arbitrary ball which intersects $F$. Let $k$ be the unique number which satisfies $t_k \leq n < t_{k+1}$. Let $j \in \{0, \ldots, N_{k+1} -1 \}$ be the unique number so
\[ 
t_k + (n_{k+1} + m_{k+1})j \leq n < t_k + (n_{k+1} + m_{k+1})(j+1).
\]
We assume that $j \geq 1$ and leave the details of the simpler case $j=0$ to the reader.
\bl \label{lemma6.11}
Suppose $\mu_{k+1} (\mc B) > 0$, then there exists (a unique choice of) $x \in \Tk$ and $i_1, \ldots, i_j \in \{1, \ldots, \# \mc S_{k+1} \}$ satisfying
\[
\nu_{k+1} (\mc B) \leq \mc L (x) \prod_{l=1}^j \exp S_{n_{k+1}} \psi(x^{k+1}_{i_l}) M_{k+1}^{N_{k+1}-j}.
\]
\el
\bp
If  $\mu_{k+1} (\mc B) > 0$, then $\Tkk \cap \mc B \neq \emptyset$. Let $z = z(x,y) \in \Tkk \cap \mc B$ where $x \in \Tk$ and $y = y(i_1, \ldots, i_{N_{k+1}}) \in \Ckk$. Let 
\[
\mc A_{x; i_1, \ldots, i_j} = \{ z(x, y(l_1, \ldots, l_{N_{k+1}} )) \in \Tkk : l_1 = i_1, \ldots, l_j = i_j \}.
\]
Suppose that $z (x^\prime, y(\ul l)) \in \mc B$. Since $\Tk$ is $(t_k, 2 \epsilon )$ separated and $n \geq t_k$, $x = x^\prime$. For $l \in \{1, 2, \ldots, j\}$, we have
\[ 
d_{n_{k+1}} (f^{t_k + (l-1)(n_{k+1} +m_{k+1})} q, x^{k+1}_{i_l}) <2\epsilon.
\]
Since $x^{k+1}_{i_l} \in \mc S_{k+1}$ and $\mc S_{k+1}$ is $(n_{k+1}, 4 \epsilon)$ separated, it follows that $l_1 = i_1, \ldots, l_j = i_j$. Thus, if $z \in \Tkk \cap \mc B$, then $z \in \mc A_{x; i_1, \ldots, i_j}$. Hence, 
\beq
\nu_{k+1} (\mc B) &\leq& \sum_{z \in \mc A_{x; i_1, \ldots, i_j}}\mc L (z) = \mc L (x) \sum_{\ul p_{k+1} : p^{k+1}_1 = i_1, \ldots, p^{k+1}_j = i_j}\mc L ( \ul p_{k+1})\\
&=& \mc L (x) \prod_{l=1}^j \exp S_{n_{k+1}} \psi (x^{k+1}_{i_l})\prod_{p=j+1}^{N_{k+1}} \sum_{l_p =1}^{\# \mc S_{k+1}} \exp S_{n_{k+1}} \psi (x^{k+1}_{l_p}),
\eeq
whence the required result.
\ep
\bl \label{lkx}
Let $x \in \Tk$ and $i_1, \ldots, i_j$ be as before. Then
\beq
\mc L(x) \prod_{l=1}^j \exp S_{n_{k+1}} \psi(x^{k+1}_{i_l}) \leq \exp \{S_n \psi(q) &+& 2n \mbox{Var}(\psi, 2 \epsilon)\\&+&\| \psi \| (\sum_{i=1}^k N_i m_i  +j m_{k+1} ) \}.
\eeq
\el
\bp
We write $x = x(\ul p_1, \ldots \ul p_k)$. Lemma \ref{dist} tells us that
\[
d_{n_i}(f^{t_{i-1} + m_{i-1} + (l-1)(m_i + n_i)}x, x^i_{p^i_{l}}) < 2 \epsilon
\] 
for all $i \in \{1, \ldots, k\}$ and all $l \in \{1, \ldots, N_i\}$ and it follows that
\[
\mc L (x) \leq \exp \{ S_{t_k} \psi (x) + t_k Var (\psi, 2 \epsilon) + \sum_{i=1}^k \| \psi \| N_i m_i \}.
\]
Similarly,
\[
\prod_{l=1}^j \exp S_{n_{k+1}} \psi(x^{k+1}_{i_l}) \leq \exp \{ S_{n-t_k}\psi (z) + (n- t_k) Var (\psi, \frac{\epsilon}{2^{k+1}}) + \| \psi \| j m_{k+1} \}.
\]
We obtain the result from these two inequalities and that $d_n(z, q) < 2\epsilon$ and $d_{t_k} (x, q) < 2\epsilon$.
\ep
The proof of the following lemma is similar to that of lemma \ref{lemma6.11}.

\bl \label{mk+p}
For any $p \geq 1$, suppose $\mu_{k+p} (\mc B) > 0$. Let $x \in \Tk$ and $i_1, \ldots, i_j$ be as before. Then every $z \in \mc T_{k+p} \cap \mc B$ descends from some point in $\mc A_{x ;i_1, \ldots, i_j}$. We have
\[
\nu_{k+p} (\mc B) \leq \mc L (x) \prod_{l=1}^j \exp S_{n_{k+1}} \psi (x^{k+1}_{i_l}) M_{k+1}^{N_{k+1}-j} M_{k+2}^{N_{k+2}} \ldots M_{k+p}^{N_{k+p}}.
\]
\el
\bl \label{8.4}
\[\mu_{k+p} (\mc B) \leq \frac{1}{\kappa_k M_{k+1}^j}\exp \left \{S_n \psi(q) + 2n Var (\psi, 2 \epsilon) +\| \psi \| (\sum_{i=1}^k N_i m_i  +j m_{k+1} ) \right \}. \]
\el
\bp
Using lemma \ref{lkx}, it follows from lemma \ref{mk+p} that
\beq
\nu_{k+p} (\mc B) \leq M_{k+1}^{N_{k+1}-j} \ldots M_{k+p}^{N_{k+p}}\exp \{S_n \psi(q) &+& 2n Var (\psi, 2 \epsilon)\\ &+& \| \psi \| (\sum_{i=1}^k N_i m_i  +j m_{k+1} ) \}.
\eeq
Since $\mu_{k+p} = \frac{1}{\kappa_{k+p}} \nu_{k+p}$ and $\kappa_{k+p}= \kappa_k M_{k+1}^{N_{k+1}} \ldots  M_{k+p}^{N_{k+p}}$, the result follows.
\ep


\begin{lemma} \label{lemma8}
For sufficiently large $n$, $\kappa_k M_{k+1}^j \geq \exp((C - 5  \gamma)n)$
\end{lemma}

\begin{proof} 
Recall that by construction $M_k \geq \exp ((C-4 \gamma) n_k)$. We have 
\beq
\kappa_k M_{k+1}^j &=& M_1^{N_1} \ldots M_k^{N_k} M_{k+1}^j\\
&\geq& \exp \{ (C-4 \gamma) (N_1 n_1 + N_2 n_2 + \ldots + N_k n_k +j n_{k+1}) \}\\
&\geq& \exp \{ (C - 5\gamma) (N_1 (n_1+ m_1) + N_2 (n_2 +m_2) + \ldots \\& & + N_k (n_k + m_k) + j (n_{k+1}+ m_{k+1}) \}\\
&=& \exp \{ (C - 5\gamma) (t_k + m_1 + j (n_{k+1}+ m_{k+1}) \} \geq \exp \{ (C - 5\gamma) n \}.
\eeq
Our arrival at the third line may require some explanation. Morally, we are able to add in the extra terms with an arbitrarily small change to the constant $s$ because $n_k$ 
is much larger than $m_k$. The reader may wish to verify this. 
\end{proof}

\bl 
For sufficiently large $n$,
\[
\limsup_{k \ra \infty} \mu_{k} (B_n (q, \frac{\epsilon}{2})) \leq \exp \{-n(C - 2Var (\psi, 2 \epsilon) - 6\gamma) + \sum_{i=0}^{n-1} \psi(f^i q)\}.
\] 
\el
\bp
By lemmas \ref{8.4} and \ref{lemma8}, for sufficiently large $n$ and any $p \geq 1$,
\beq
\mu_{k+p} (\mc B) &\leq& \frac{1}{\kappa_k M_{k+1}^j}\exp \left \{S_n \psi(q) + 2n V +\| \psi \| (\sum_{i=1}^k N_i m_i  +j m_{k+1} ) \right \}\\
&\leq& \frac{1}{\kappa_k M_{k+1}^j} \exp \left \{S_n \psi(q) + n \left(2V + \gamma ) \right) \right \}\\
&\leq& \exp\{-n(C- 6\gamma - 2V)) + S_n \psi (q)\},
\eeq
where $V = \mbox{Var}(\psi, 2 \epsilon)$. Our arrival at the second line is because $n_k$ is much larger than $m_k$.
\ep
Applying the Generalised Pressure Distribution Principle, 
we have
\[
P_F (\psi, \epsilon) \geq C - 2\mbox{Var} (\psi, 2 \epsilon) - 6\gamma.
\] 
Recall that $\epsilon$ was chosen sufficiently small so $\mbox{Var}(\psi, 2 \epsilon) < \gamma$. It follows that
\[
P_{\widehat X_\varphi}(\psi, \epsilon) \geq P_F (\psi, \epsilon) \geq C - 8 \gamma.
\]
Since $\gamma$ and $\epsilon$ were arbitrary, the proof of theorem \ref{theorem5} is complete.  
\subsection{Modification of the construction to obtain theorem \ref{theorem1}} \label{modif}
Let us fix a small $\gamma > 0$. Let $\mu_1$ be ergodic and satisfy $h_{\mu_1} + \int \psi d \mu_1 > C - \gamma /2$. Let $\nu \in \mc M_f^e (X^\prime)$ satisfy $\int \varphi d {\mu_1} \neq \int \varphi d \nu$. Let $\mu_2 = t_1 \mu_1 + t_2 \nu$ where $t_1 + t_2 =1$ and $t_1 \in (0,1)$ is chosen sufficiently close to $1$ so that $h_{\mu_2} + \int \psi d \mu_2 > C -  \gamma$.
Choose $\delta > 0$ sufficiently small so
\[
\left | \int \varphi d \mu_1 - \int \varphi d \mu_2 \right | > 8 \delta.
\]
Choose a strictly decreasing sequence $\delta_k \ra 0$ with $\delta_1 < \delta$. For $k$ odd, we proceed as before, choosing a strictly increasing sequence $l_k \ra \infty$ so the set
\[
Y_k := \left \{ x \in X^\prime : \left | \frac{1}{n} S_n \varphi (x) - \int \varphi d\mu_{1} \right | < \delta_k \mbox{ for all } n \geq l_k\right \}
\]
satisfies $\mu_{1} (Y_k) > 1 - \gamma$ for every $k$. For $k$ even, we define $Y_{k, 1} := Y_{k-1}$ and find $l_k > l_{k-1}$ so that each of the sets
\[
Y_{k, 2} := \left \{ x \in X^\prime : \left | \frac{1}{n} S_n \varphi (x) - \int \varphi d \nu \right | < \delta_k \mbox{ for all } n \geq l_k\right \}
\]
satisfies $\nu (Y_{k,2}) > 1 - \gamma$.
The proof of the following lemma is similar to that of lemma \ref{lemma1}.
\begin{lemma} 
For any sufficiently small $\epsilon > 0$ and $k$ even, we can find a sequence $\hat n_k \ra \infty$ so  $[t_i \hat n_k] \geq l_{k}$ for $i=1,2$ and sets $\mc S^i_k$ so that $\mc S^i_k$ is a $([t_i \hat n_k], 4 \epsilon)$ separated set for $Y_{k,i}$ with $M^i_k := \sum_{x \in \mc S^i_k} \exp \left \{ \sum_{j=0}^{n_k -1} \psi(f^{j}x)\right \}$ satisfying
\[
M^1_k \geq \exp( [t_1 \hat n_k] (h_{\mu_1} + \int \psi d \mu_1  - 4 \gamma)),
\]
\[
M^2_k \geq \exp( [t_2 \hat n_k] (h_{\nu} + \int \psi d \nu  - 4 \gamma)).
\]
Furthermore, the sequence $\hat n_k$ can be chosen so that $\hat n_k \geq  2^{m_{k}}$ where $m_k = m(\epsilon/2^k)$ is as in the definition of specification. 
\end{lemma}
We now use the specification property to define the set $\mc S_k$ as follows. For $i=1,2$, let $y_i \in \mc S^i_k$ and define $x = x(y_1, y_2)$ to be a choice of point which satisfies
\[
d_{[t_1 \hat n_k]}(y_1, x) < \frac{\epsilon}{2^k}\mbox{ and } d_{[t_2 \hat n_k]}(y_2, f^{[t_1 \hat n_k] + m_k} x) < \frac{\epsilon}{2^k}.
\]
Let $\mc S_k$ be the set of all points constructed in this way. Let $n_k = [t_1 \hat n_k] + [t_2 \hat n_k] + m_k$. Then $n_k$ is the amount of time for which the orbit of points in $\mc S_k$ has been prescribed and we have $n_k/\hat n_k \ra 1$. We note that $\mc S_k$ is $(n_k, 4 \epsilon)$ separated and so $\# \mc S_k = \# \mc S_k^1 \# S_k^2$. Let $M_k = M_k^1 M_k^2$. Given our new construction of $\mc S_k$, the rest of our constuction goes through unchanged.
\subsection{Modification to the proof}
For every $x \in \mc S_k$,
\beq
|S_{n_k} \varphi (x) - n_k \int \varphi d \mu_2 | &\leq& |S_{[t_1 \hat n_k]} \varphi (x)- [t_1 \hat n_k] \int \varphi d \mu_1| + m_k \|\varphi \| \\&+&|S_{[t_2 \hat n_k]} \varphi (f^{[t_1 n_k] + m_k}x) - [t_2 \hat n_k] \int \varphi d \nu |
\eeq
It follows that $|\frac{1}{n_k} S_{n_k} \varphi (x) - \int \varphi d \mu_2| \ra 0$. This observation allows us to modify the proof of lemma \ref{lemma6} and ensures that our construction still gives rise to points in $\widehat X_\varphi$.
We have for sufficiently large $n_k$,
\beq
M_k &\geq& \exp \{[t_1 \hat n_k](h_{\mu_1} + \int \psi d \mu_1  - 4 \gamma) + [t_2 \hat n_k](h_{\nu} + \int \psi d \nu  - 4 \gamma)\}\\&\geq& \exp\{ (1 - \gamma) \hat n_k (t_1(h_{\mu_1} + \int \psi d \mu_1) + t_2(h_{\nu} + \int \psi d \nu)  - 4 \gamma)\}\\
&\geq& \exp (1 - \gamma)^2 n_k (h_{\mu_2} + \int \psi d \mu_2 - 4 \gamma )
\geq \exp (1 - \gamma)^2 n_k (C - 5 \gamma).
\eeq
Since $\gamma$ was arbitrary, this observation allows us to modify the estimates in lemma \ref{lemma8} to cover this more general construction. 
\section{Examples} \label{examp}
\subsection{Standard examples}
We recall that any factor of a topologically mixing shift of finite type has the specification property and thus our result applies.  
Bowen's specification theorem tells us that a compact locally maximal hyperbolic set of a topologically mixing diffeomorphism $f$ has the Bowen specification property. In particular, 
our result applies to topologically mixing Anosov diffeomorphisms (which include any Anosov diffeomorphism of a compact connected manifold whose wandering set is empty). 

\subsection{The Manneville-Pomeau family of maps}
Let $I = [0,1]$. The MP family of maps, parametrised by $\alpha \in (0, 1)$ are given by
\[
f_\alpha : I \mapsto I, f_\alpha (x) = x + x^{1+\alpha} \mod 1.
\]
Considered as a map of $S^1$, $f_\alpha$ is continuous. Since $f_\alpha^{\prime} (0) = 1$, the system is not uniformly hyperbolic and thus the results of \cite{BS} do not apply. However, since the MP maps are all topologically conjugate to a full shift on two symbols, they satisfy the specification property and thus theorem \ref{theorem1} applies. 

\subsection{Beyond Symbolic Dynamics}
As remarked in the introduction, by the Blokh theorem, any topologically mixing interval map satisfies specification. For example, Jakobson \cite{Ja} showed that for a set of parameter values of positive Lebesgue measure in $[0, 4]$, the logistic map $f_\lambda (x) = \lambda x (1 - x)$ is topologically mixing.  

Lind \cite{Li} showed that a quasi-hyperbolic toral automorphism satisfies specification but not Bowen specification iff the matrix representation of the automorphism in Jordan normal form admits no $1$'s off the diagonal in the central direction. Such maps cannot be factors of topologically mixing shifts of finite type or they would inherit the Bowen specification property. 

Theorems 17.6.2 and 18.3.6 of \cite{KH} ensure that the geodesic flow of any compact connected Riemannian manifold of negative sectional curvature is topologically mixing and Anosov. The specification theorem for flows (proved in \cite{Bo4}) ensures that such a flow has the specification property 18.3.13 of \cite{KH}. It is easy to see that the time-$t$ map of a flow with the specification property satisfies our specification property \ref{3a}. We conclude that our results apply to the time-$t$ map of the geodesic flow of any compact connected Riemannian manifold of negative sectional curvature.

\section{Application to Suspension Flows}
We apply our main result to suspension flows. Let $f : X \mapsto X$ be a homeomorphism of a compact metric space $(X, d)$. We consider a continuous roof function $\rho : X \mapsto (0, \infty)$. 
We define the suspension space to be
\[
X_\rho = \{ (x, s) \in X \times \IR : 0 \leq s \leq \rho(x) \},
\]
where $(x, \rho(x))$ is identified with $(f(x), 0)$ for all $x$. Alternatively, we can define $X_\rho$ to be $X \times [0, \infty)$, 
quotiented by the equivalence relation $(x, t) \sim (y, s)$ iff $(x,t) = (y,s)$ or there exists $n \in \IN$ so $(f^n x, t - \sum_{i=0}^{n-1} \rho (f^i x)) = (y, s)$ or $(f^{-n} x, t + \sum_{i=1}^{n} \rho (f^{-i}x)) = (y,s)$. 
Let $\pi$ denote the quotient map from $X \times [0, \infty)$ to $X_\rho$. We extend the domain of definition of $\pi$ to $X \times ( - \inf \rho, \infty)$ by identifying points of the form $(y, - t)$ with $(f^{-1} y, \rho (y) -t)$ for $t \in (0, \inf \rho)$. We write $(x,s)$ in place of $\pi(x, s)$ when $\inf \rho < s < \rho (x)$.
We define the flow $ \Psi = \{ g_t \}$ on $X_\rho$ by
\[
g_t (x, s) = \pi(x, s+t).
\]
To a function $\Phi :X_\rho \mapsto \IR$, we associate the function $\varphi: X \mapsto \IR$ by $\varphi(x) = \int_0^{\rho(x)} \Phi (x, t) dt$. Since the roof function is continuous, when $\Phi$ is continuous, so is $\varphi$. For $\mu \in \mathcal{M}_{f} (X)$, we define the measure $\mu_\rho$ by
\[
\int_{X_\rho} \Phi d \mu_\rho = \int_X \varphi d \mu / \int \rho d \mu
\]
for all $\Phi \in C (X_\rho)$, where $\varphi$ is defined as above. We have $\Psi$-invariance of $\mu_\rho$ (ie. $\mu (g_t ^{-1} A) = \mu (A)$ for all $t \geq 0$ and measurable sets $A$). The map $\mc R : \mathcal{M}_{f} (X) \mapsto \mathcal{M}_{\Psi} (X_\rho)$ given by $\mu \mapsto \mu_\rho$ is a bijection. It is verified in \cite{PP} that $h_{\mu_\rho} = h_\mu / \int \rho d \mu$ and hence, 
\[
\htop (\Psi) = \sup \{ h_\mu : \mu \in \mathcal{M}_{\Psi} (X_\rho) \} = \sup \left \{\frac{ h_\mu }{ \int \rho d \mu} : \mu \in \mathcal{M}_{f} (X) \right \},
\]
where $\htop(\Psi)$ is the topological entropy of the flow. Abramov's theorem states that $\htop (\Psi)$ is the unique solution to the equation $P_X^{classic} (-s \rho) = 0$. We use the notation  $h_{top} ( Z, \Psi )$ for topological entropy of a non-compact subset $Z \subset X_\rho$ with respect to $\Psi$ (defined below).
We define
\[
\widehat X_\rho = \{ (x,s) \in \xr : \lim_{T \ra \infty} \frac{1}{T} \int_0^T \Phi(g_t (x,s)) d t \mbox{ does not exist } \}.
\]
By the ergodic theorem for flows, $\mu (\widehat X_\rho) = 0$ for any $\mu \in \mathcal{M}_{\Psi} (X_\rho)$. Our main result on suspension flows is the following (the proof is at the end of the section).
\bt \label{susflow}
Let $(X,d)$ be a compact metric space and $f:X \mapsto X$ be a homeomorphism with the specification property. Let $\rho : X \mapsto (0, \infty)$ be continuous. Let $(\xr, \Psi)$ be the corresponding suspension flow over $X$. Assume that $\Phi : \xr \mapsto \IR$ is continuous and satisfies $\inf_{\mu \in \mathcal{M}_{\Psi} (X_\rho)} \int \Phi d \mu < \sup_{\mu \in \mathcal{M}_{\Psi} (X_\rho)} \int \Phi d \mu$. Then $h_{top} ( \widehat \xr, \Psi ) = h_{top} (\Psi)$.
\et
We remark that the flow $\Phi$ may not satisfy specification itself. For example, when $\rho$ is a constant fuction, $\Phi$ is not even topologically mixing.

\subsection{Topological entropy for flows as a characteristic of dimension type} \label{tef}
Let $Z \subset X$ be an arbitrary Borel set, not necessarily compact or invariant. Let $ \Psi = \{ \psi_t \}$ be a flow on $X$. We consider finite and countable collections of the form $\Gamma = \{B_{t_i}(x_i, \epsilon) \}_i$, where $t_i \in (0, \infty)$, $x_i \in X$ and
\[
B_{t}(x, \epsilon) = \{ y \in X : d( \psi_s (x), \psi_s (y)) < \epsilon \mbox{ for all } s \in [0, t)\}.
\]
For $\alpha \in \IR$, we define the following quantities:
\[
Q(Z,\alpha, \Gamma) = \sum_{B_{t_i}(x_i, \epsilon) \in \Gamma} \exp \left(-\alpha t_i \right),
\]
\[
M(Z, \alpha, \epsilon, T) = \inf_{\Gamma} Q(Z,\alpha, \Gamma),
\]
where the infimum is taken over all finite or countable collections of the form $\Gamma = \{ \mathit{B}_{t_i}(x_i, \epsilon) \}_i$ with $x_i \in X$ such that $\Gamma$ covers Z and $t_i \geq T$ for all $i = 1, 2, \ldots$. Define
\[
m(Z, \alpha, \epsilon) = \lim_{T \rightarrow \infty} M(Z, \alpha,\epsilon, T).
\]
The existence of the limit is guaranteed since the function $M(Z, \alpha,\epsilon, T)$ does not decrease with $T$. By standard techniques, we can show the existence of
\begin{displaymath}
\htop (Z, \epsilon) := \inf \{ \alpha : m(Z, \alpha, \epsilon) = 0\} = \sup \{ \alpha :m(Z, \alpha, \epsilon) = \infty \}.
\end{displaymath}
\begin{definition}
The topological entropy of $Z$ with respect to $\Psi$ is given by
\[
\htop (Z, \Psi) = \limn \htop (Z, \epsilon).
\]
\ed
\subsection{Properties of suspension flows}

The following lemma is similar to one given in \cite{BS2}. 
\bl \label{tf}
Let $(X,d)$ be a compact metric space and $f:X \mapsto X$ be a homeomorphism. Let $\rho : X \mapsto (0, \infty)$ be continuous. Let $(\xr, \Psi)$ be the corresponding suspension flow over $X$. Let $\Phi: \xr \mapsto \IR$ be continuous and $\varphi: X \mapsto \IR$ be given by $\varphi(x) = \int_0^{\rho(x)} \Phi (x, t) dt$. We have
\[
\liminf_{T \ra \infty} \frac{1}{T} \int_0^T \Phi(g_t (x,s)) d t = \liminf_{n \ra \infty} \frac{S_n \varphi (x)}{S_n \rho (x)},
\]
\[
\limsup_{T \ra \infty} \frac{1}{T} \int_0^t \Phi(g_t (x,s)) d t = \limsup_{n \ra \infty} \frac{S_n \varphi (x)}{S_n \rho (x)},
\]
\[
\widehat \xr = \{ (x, s) : \limn \frac{S_n \varphi (x)}{S_n \rho (x)} \mbox{ does not exist}, 0\leq s < \rho (x) \}.
\]
\el
\bp
Fix $\gamma > 0$. Given $T > 0$, let $n$ satisfy $S_n \rho(x) \leq T <S_{n+1} \rho (x)$. It follows that $1- \frac{\| \rho \|}{T} \leq \frac{S_n \rho (x)}{T} \leq 1$. Assume $T$ is sufficiently large that $2 T^{-1} \|\rho\| \| \Phi\| < \gamma$.
We note that
\beq
\int_0^T \Phi(g_t (x,s)) d t &\leq& \sum_{i=0}^{n-1} \int^{\rho (f^{i} x)}_0 \Phi(f^i x,t) d t +2\|\rho\| \| \Phi\| \\
&=& S_{n} \varphi (x) +2\|\rho\| \| \Phi\|,
\eeq
and so 
\beq
\frac{1}{T} \int_0^T \Phi(g_t (x,s)) d t &\leq& \frac{S_n \rho(x)}{T}  \frac{S_{n} \varphi (x)}{S_n \rho (x)} +\frac{2}{T}\|\rho\| \| \Phi\| \\ &\leq& \frac{S_{n} \varphi (x)}{S_n \rho (x)} + \gamma.
\eeq
The result follows from this and a similar calculation for the opposite inequality.
\ep
As the lemma suggests, our result on $\widehat X_\rho$ will follow from a corresponding result about the set
\begin{equation} \label{xpr}
\widehat X (\varphi, \rho) := \left \{ x \in X : \limn \frac{S_n \varphi (x)}{S_n \rho (x)} \mbox{ does not exist}\right \}.
\end{equation}
\bl
Under our assumptions, the following are equivalent: 

(a) $\widehat X_\rho \neq \emptyset$; (b) $\widehat X (\varphi, \rho) \neq \emptyset$; 

(c) $\inf_{\mu \in \mathcal{M}_{\Psi} (X_\rho)} \int \Phi d \mu < \sup_{\mu \in \mathcal{M}_{\Psi} (X_\rho)} \int \Phi d \mu$;

(d) $\inf_{\mu \in \mathcal{M}_{f} (X)} \int \varphi d \mu / \int \rho d \mu < \sup_{\mu \in \mathcal{M}_{f} (X)} \int \varphi d \mu / \int \rho d \mu$;

(e) $\inf_{\mu \in \mathcal{M}^e_{f} (X)} \int \varphi d \mu / \int \rho d \mu < \sup_{\mu \in \mathcal{M}^e_{f} (X)} \int \varphi d \mu / \int \rho d \mu$;

(f) $S_n \varphi /S_n \rho$ does not converge (uniformly or pointwise) to a constant; 





(g) $\frac{1}{T} \int_0^T \Phi(g_t) d t$ does not converge (uniformly or pointwise) to a constant;

Let $\varphi_T (x) := \int_0^T \Phi(g_t x) d t$.

(h) There exists $T$ such that $\varphi_T \notin  \bigcup_{c \in \IR} \overline{Cob (X_\rho, g_T, c)}$, i.e $\varphi_T$ is not in the closure of the coboundaries for the time-$T$ map of the flow; 

(i) For all $T$, $\varphi_T \notin \bigcup_{c \in \IR} \overline{Cob (X_\rho, g_T, c)}$.
\el
\bp
First we note that (d) $\iff$ (e) $\iff$ (f) is similar to the proof of the analogous statements in lemma \ref{equiv}.
For (c) $\Rightarrow$ (d), let $\mu_1, \mu_2 \in \mathcal{M}_{\Psi} (X_\rho)$ satisfy $\int \Phi d \mu_1 < \int \Phi d \mu_2$. Let $v_i = \mc R^{-1} \mu_i$ for $i = 1, 2$. By definition, $\int \varphi d v_i / \int \rho d v_i = \int \Phi d \mu_i$ for $i = 1, 2$ and so $\int \varphi d v_1 / \int \rho d v_1 < \int \varphi d v_2 / \int \rho d v_2$. (d) $\Rightarrow$ (c) is similar. (f) $\iff$ (g) follows from lemma \ref{tf}.


We show (c) $\iff$ (h) $\iff$ (i). We define bijections $\mc R_T : \mathcal{M}_{g_T} (X_\rho) \mapsto \mathcal{M}_{\Psi} (X_\rho)$ by 
\[
\int_{X_\rho} \Phi d \mc R_T(\mu) = \frac{1}{T} \int_X \varphi_T d \mu
\]
for all $\Phi \in C (X_\rho)$, where $\varphi_T (x) := \int_0^T \Phi(g_t x) d t$. A similar argument to that of (c) $\iff$ (d) and an appliction of lemma \ref{equiv} gives the desired results.

(a) $\Rightarrow$ (g), (b) $\Rightarrow$ (f), (b) $\Rightarrow$ (a) are trivial. (d) $\Rightarrow$ (b) is a consequence of theorem \ref{ratio}, so we omit the proof.
\ep
We remark that if $\varphi \in \overline{Cob (X, f, 0)}$ or $\varphi - \rho \in \overline{Cob (X, f, 0)}$, then $S_n \varphi / S_n \rho$ converges uniformly to a constant and so $\widehat X_\rho = \emptyset$.

\subsection{A generalisation of the main theorem}
To prove theorem \ref{susflow}, we require the following generalisation of theorem \ref{theorem0.1}.
\begin{theorem} \label {ratio}
Let $(X,d)$ be a compact metric space and $f:X \mapsto X$ be a continuous map with specification. Let $\varphi, \psi \in C(X)$ and $\rho : X \mapsto (0, \infty)$ be continuous with $\inf_{\mu \in \mathcal{M}_{f} (X)} \int \varphi d \mu / \int \rho d \mu < \sup_{\mu \in \mathcal{M}_{f} (X)} \int \varphi d \mu / \int \rho d \mu$. Let $\widehat X (\varphi, \rho)$ be defined as in (\ref{xpr}). We have $P_{\widehat X(\varphi, \rho)} (\psi) = P_X^{classic} (\psi)$.
\end{theorem}
\bp
We require only a small modification to the proof of theorem \ref{theorem1}. We replace the family of sets defined at (\ref{5}) by the following:
\[
Y_k := \left \{ x \in X : \left | \frac{S_n \varphi (x)}{S_n \rho (x)} - \frac{\int \varphi d\mu_{\rho(k)}}{\int \rho d \mu_{\rho(k)}}  \right | < \delta_k \mbox{ for all } n \geq l_k\right \}
\]
chosen to satisfy $\mu_{\rho(k)} (Y_k) > 1 - \gamma$ for every $k$. This is possible by the ratio ergodic theorem. The rest of the proof requires only superficial modifications.
\ep
\subsection{The relationship between entropy of a suspension flow and pressure in the base}
The natural metric on $X_\rho$ is the Bowen-Walters metric. The appendix of \cite{BS2} contains a study of dynamical balls taken with respect to this metric when the roof function is H\"older. We assume only continuity of $\rho$. 
When $\rho$ is non-constant, computations involving this metric are rather unwieldy, particularly when no regularity of the roof function is assumed. We sidestep this problem by making the following definitions. Let $(x, s) \in X_\rho$ with $0 \leq s < \rho (x)$. We define the horizontal segment of $(x,s)$ to be $\{(y, t): y \in X, 0 \leq t < \rho (y), t = \rho (y) s \rho(x)^{-1} \}$ and the horizontal ball of radius $\epsilon$ at $(x,s)$ to be
\[
B^H((x,s), \epsilon) := \{(y, \frac{s}{\rho(x)} \rho(y)) : (1-\frac{s}{\rho(x)}) d(x,y) + \frac{s}{\rho(x)} d(fx, fy) < \epsilon \}.
\]
We define
\[
B((x,s), \epsilon) = \bigcup_{t : |s-t| < \epsilon} B^H((x,t), \epsilon), \] 
\[B_T ((x,s), \epsilon) = \bigcap_{t=0}^T g_{-t} B(g_t(x,s), \epsilon).\]
We are abusing notation, since $B((x,s), \epsilon)$ is not a ball in the Bowen-Walters metric. We can consider covers by sets of the form $B_T ((x,s), \epsilon)$ in the definition of topological pressure in place of covers consisting of dynamical balls. This is because one can verify that there exists constants $C_1, C_2 > 0$ such that the metric ball of radius $C_1 \epsilon$ at $(x,s)$ is a subset of $B((x,s), \epsilon)$, that a set of diameter $\epsilon$ is contained in some set $B((x,s), C_2 \epsilon)$ for sufficiently small $\epsilon$, that $B((x,s), \epsilon)$ is open and as $\epsilon \ra 0$, $diam (\{B((x,s), \epsilon): (x, s) \in X_\rho \}) \ra 0$. Diameter and topology are taken with respect to the Bowen-Walters metric. 
\bl
Let $(y,s) \in X \times (-\inf \rho, \infty)$ and suppose $\pi(y, s) \in B((x, \delta), \epsilon)$, where $|\delta| \leq \epsilon < \inf \rho /4$. Then for $\epsilon$ sufficiently small there exists $n \in \IN$ such that \[(y, s) \sim (f^n y, s -S_n \rho (y)),
|s- S_n \varphi (y)| < K \epsilon \mbox{ and } d(x, f^ny) < K \epsilon,
\] 
where $K = 4\| \rho \| / \inf \rho$ and $K \epsilon < \inf \rho$.
\el
\bp
Suppose $(y, s) \in B^H((x, \gamma), \epsilon)$ for some $\gamma$ with $0 \leq |\gamma| < 2\epsilon$. Then $s = \gamma \rho (y) \rho (x)^{-1}$. Therefore, $s < 2\epsilon \|\rho\|/ \inf \rho$. We have
\[
(1 - \frac{\gamma}{\rho(x)}) d(x,y) + \frac{\gamma}{\rho(x)} d(fx, fy) < \epsilon.
\]
Thus $(1 - \frac{\gamma}{\rho(x)}) d(x,y) < \epsilon$. Rearranging, we have $d(x,y) <\epsilon \rho(x) (\rho(x) - \gamma)^{-1} < K \epsilon$. For $-\epsilon <\gamma < 0$, we apply a similar argument. 
Now assume $\pi(y, s) \in B((x, \delta), \epsilon)$. Then $\pi(y, s)$ has a unique representation $(y^\prime, s^\prime)$ with $|s^\prime|< 2\epsilon$ and $y^\prime = f^n y$. We apply the previous argument to $(y^\prime, s^\prime)$.
\ep
\bl \label{susballs}
Suppose $| s| < \epsilon$ and $S_n \rho (x) \leq T <S_{n+1} \rho (x)$, then 
\[
B_T((x,s), \epsilon) \subset B_{n} (x, K \epsilon)) \times (-K \epsilon, K \epsilon).
\]
\el
\bp
Let $(y, t) \in B_T((x,s), \epsilon)$, with $|t| < K\epsilon$. Then $d(x, y) < K \epsilon$. Let $t_i$ satisfy $s+t_i = S_i \rho (x)$ for $i =1, \ldots n$. Then $g_{t_i} (y,t) \in B((f^{i-1}x, 0), \epsilon)$. Applying the previous lemma, we have $d(f^n y, f^{i-1}x) < K \epsilon$ for some $n \in \IN$. Furthermore, we must have $n=i-1$. Suppose not, 
then for some time $\tau \in [0, S_i\rho(x))$, $g_\tau (y, t) \notin B(g_\tau(x,s), \epsilon)$, which is a contradiction. This implies that $y \in B_{n} (x, K\epsilon)$. 
\ep

\bt \label{basepressure}
Let $(X,d)$ be a compact metric space and $f:X \mapsto X$ be a homeomorphism. Let $\rho : X \mapsto (0, \infty)$ be continuous. Let $(\xr, \Psi)$ be the corresponding suspension flow over $X$. For an arbitrary Borel set $Z \subset X$, define $Z_\rho := \{ (z, s) : z \in Z, 0 \leq s < \rho(s) \}$. Let $\beta$ be the unique solution to the equation $P_Z ( - t \rho) = 0$. Then $\htop (Z_\rho, \Psi) \geq \beta$. 
\et
\bp
The function $t \ra P_Z ( - t \rho)$ is continuous and decreasing. Since $P_Z (0) \geq 0$, it follows that there exists a unique solution to the equation $P_Z ( - t \rho) = 0$. We assume $P_Z (- \beta \varphi) > 0$ and show $\htop (Z_\rho, \Psi) \geq \beta$. 
Let $\epsilon > 0$ be arbitrary and sufficiently small so lemma \ref{susballs} applies and $P_{Z} (- \beta \varphi, \epsilon) > 0$. Choose $\Gamma = \{B_{t_i}((x_i, s_i), \epsilon) \}$ covering $Z_\rho$ with $t_i \geq T$. 
Take the subcover $\Gamma^\prime$ of $\Gamma$ which covers $Z \times \{0\}$, and assume without loss of generality that $|s_i| < \epsilon$. Let $m_i$ be the unique number so $S_{m_i}\rho (x) \leq t_i < S_{m_i+1} \rho(x)$. Let $m(\Gamma^\prime) = \inf m_i$ obtained in this way. Then $m(\Gamma^\prime) \geq \|\rho\|^{-1} (T - \| \rho \|)$ and thus as $T$ tends to infinity so does $m(\Gamma^\prime)$. Let $\Gamma^{\prime \prime} = \{ B_{m_{i}} (x_i, K \epsilon) \} : B_{t_i}((x_i, s_i), \epsilon) \in \Gamma^\prime \}$. By lemma \ref{susballs}, $B_{m_{i}} (x_i, K \epsilon) \times (-K\epsilon, K\epsilon)$ covers $Z \times \{0\}$ and if we assume $\epsilon$ was chosen sufficiently small, then $\Gamma^{\prime \prime}$ is a cover for $Z$. 
\beq
 Q(Z\times\{0\}, \beta, \Gamma^\prime) &\geq& \sum_{B_i \in \Gamma^\prime} \exp -\beta (S_{m_i} \rho (x_i) +\| \rho\|)\\
&\geq& \sum_{B_i \in \Gamma^{\prime \prime}} \exp - \beta (\sup_{y \in B_i} S_{m_i} \rho (y) +\| \rho\| + \mbox{Var}(\rho, K \epsilon))\\
&=& \exp\{ - \beta (\mbox{Var}(\rho, K \epsilon) + \|\rho\|)\} Q(Z, 0, \Gamma^{\prime \prime}, -\beta \rho)\\
&\geq& \exp\{ - \beta (\mbox{Var}(\rho, K \epsilon) + \|\rho\|)\} M(Z, 0, m(\Gamma^\prime), -\beta \rho)\\
&\geq& 1,
\eeq
if $T$ and hence $m(\Gamma^\prime)$ are chosen to be sufficiently large. We have 
\[Q(Z_\rho, \beta, \Gamma) \geq Q(Z\times \{0\}, \beta, \Gamma^\prime)
\] and since $\Gamma$ was arbitrary, we have $M(Z_\rho, \beta, T- \|\rho\|, \epsilon) \geq 1$ and hence $\htop (Z_\rho, \Psi, \epsilon) \geq \beta$.
\ep
\subsection{Proof of Theorem \ref{susflow}}
Given the results we have proved so far, theorem \ref{susflow} follows easily. By lemma \ref{tf}, $\widehat X_{\rho} = Z_\rho$, where $Z= \widehat X(\varphi, \rho)$. We recall that $\htop (\Psi)$ is the unique number satisfing $P_X^{classic} (- t \rho) = 0$. By theorem \ref{ratio}, $P_Z (- t \rho)= P^{classic}_X (-t \rho)$ for all $t \in \IR$, and so $\htop (\Psi)$ is the unique number such that $P_Z ( - t \rho) = 0$. Applying theorem \ref{basepressure}, our result follows.
\section*{Acknowledgements}
This work constitutes part of my PhD, which is supported by the EPSRC. I would like to thank my supervisors Mark Pollicott and Peter Walters for many useful discussions and reading draft versions of this work, for which I am most grateful.

\bibliographystyle{amsplain}
\bibliography{fullpressure}

\end{document}